\documentclass[11pt]{article}

\usepackage{amsmath}
\usepackage{amsthm}
\usepackage{amsfonts} 
\usepackage[pdftex]{graphicx} 
\usepackage[english]{babel} 
\usepackage[pdftex,linkcolor=black,pdfborder={0 0 0}]{hyperref} 
\usepackage{calc} 

\usepackage{color}
\usepackage{enumitem} 
\usepackage{tikz}
\usepackage{indentfirst}
\usepackage[all]{nowidow}
\usepackage[protrusion=true,expansion=true]{microtype}
\usepackage{lipsum} 

\usepackage{tikz}
\usepackage[all,knot]{xy}
\usetikzlibrary{braids}

\hypersetup{ 	
	pdfsubject = {},
	pdftitle = {},
	pdfauthor = {}
}

\theoremstyle{plain}
\newtheorem{theorem}{Theorem}
\newtheorem*{theorem*}{Previous Work}

\newtheorem{lemma}[theorem]{Lemma}

\theoremstyle{definition}
\newtheorem{definition}[theorem]{Definition}

\theoremstyle{remark}
\newtheorem*{remark}{Remark}

\makeatletter
\newcommand{\address}[1]{\gdef\@address{#1}}
\newcommand{\email}[1]{\gdef\@email{\url{#1}}}
\newcommand{\@endstuff}{\par\vspace{\baselineskip}\noindent\small
\begin{tabular}{@{}l}\scshape\@address\\\textit{E-mail address:} \@email\end{tabular}}
\AtEndDocument{\@endstuff}
\makeatother

\title{An explicit description of $(1,1)$ L-space knots, and non-left-orderable surgeries}

\author{Zipei Nie}
\address{Nine-Chapter Lab, Huawei}
\email{niezipei@huawei.com}
\begin{document} 
\maketitle
\abstract{Greene, Lewallen and Vafaee characterized $(1,1)$ L-space knots in $S^3$ and lens space in the notation of coherent reduced $(1,1)$-diagrams. We analyze these diagrams, and deduce an explicit description of these knots. With the new description, we prove that any L-space obtained by Dehn surgery on a $(1,1)$-knot in $S^3$ has non-left-orderable fundamental group.} 

\section{Introduction}

	An L-space is a rational homology $3$-sphere with minimal Heegaard Floer homology, that is, $\dim\widehat{HF} =|H_1(Y)|$. A nice topological property of L-spaces is that \cite{OS} they do not admit co-orientable taut foliations, and its converse statement is only partially verified. Another conjectural property of L-spaces is the non-left-orderability of fundamental groups \cite{BGW}, that is, there does not exist a total order $\le$ on the fundamental group such that $g\le h$ implies $fg\le fh$. Although we have multiple computational tools, the Heegaard Floer data is not easy to utilize. Therefore, a better characterization of L-spaces would be helpful.

	One way to construct L-spaces is via Dehn surgeries. A knot $K$ is called an L-space knot, if it admits an L-space surgery. It is a positive (resp. negative) L-space knot if it admits a positive (resp. negative) L-space surgery. The Dehn surgery along a nontrivial positive L-space knot $K$ in $S^3$ with slope $\frac{p}{q}$ yields an L-space if and only if $\frac{p}{q}\ge 2g(K)-1$ \cite{OS10}. Similar results also hold for knots in other L-spaces \cite{RR15}.

	For a closed orientable $3$-manifold $Y$, we say that a knot $K$ in $Y$ is a $(g,b)$-knot, if there exists a Heegaard splitting $Y= U_0 \cup U_1$ of genus $g$, such that each of $K\cap U_0$ and $K\cap U_1$ consists of $b$ trivial arcs. The family of $(1,1)$-knots (also called $1$-bridge torus knot in the literature) in the $3$-sphere and lens spaces is widely studied. The knot Floer invariants arises diagrammatically \cite{GMM05,Hed11,Ras05} if we can find a $(1,1)$-decomposition.

	A $(1,1)$-diagram for a $(1,1)$-knot $K$ in the three-sphere or lens space $Y$ a doubly-pointed Heegaard diagram $(\Sigma,\alpha,\beta,w,z)$, which consists of two simple closed curves $\alpha$ and $\beta$ on the torus $\Sigma$ and two basepoints $w$ and $z$ in $\Sigma-\alpha-\beta$. The diagram $(\Sigma,\alpha,\beta,w,z)$ is called reduced if each bigon contains a basepoint. In this case, the diagram can be specified \cite{Ras05} by four parameters $p,q,r,s$. Via successive isotopies to removing empty bigons, every $(1,1)$-knot has a reduced $(1,1)$-diagram. In \cite{GLV}, Greene, Lewallen and Vafaee established the following criterion to determine whether a reduced $(1,1)$-diagram represents an L-space knot.

\begin{theorem*}\label{GLV-coherent} \cite[Theorem 1.2]{GLV}
A reduced $(1,1)$-diagram represents an L-space knot if and only if it is coherent, that is, there exist orientations on $\alpha$ and $\beta$ that induce coherent orientations on the boundary of every embedded bigon $(D,\partial D)\subseteq (\Sigma, \alpha \cup \beta)$. It represents a positive or negative L-space knot according to the sign of $\alpha \cdot \beta$ with coherent orientation.
\end{theorem*}

Building on their work, we describe the family of $(1,1)$ L-space knots explicitly as follows.

\begin{theorem}\label{main-description}
Let $Y=U_0\cup_{\Sigma} U_1$ be a genus one Heegaard splitting of a three-sphere or lens space, with standard geometry. A knot in $Y$ is a $(1,1)$ L-space knot if and only if it is isotopic to a union of three arcs $\rho\cup \tau_0 \cup \tau_1$, such that    
	\begin{enumerate}[label=(\alph*)]
		\item $\rho$ is a geodesic of $\Sigma$;
		\item $\tau_0$ is properly embedded in some meridional disk of $U_0$;
		\item $\tau_1$ is properly embedded in some meridional disk of $U_1$.
	\end{enumerate}
\end{theorem}

	Note that, if $\tau_0$ or $\tau_1$ is of length zero, then by definition, the knot is a $1$-bridge braid in $Y$. The study of $1$-bridge braids originates from the classification of knots in a solid torus with nontrivial solid torus surgeries \cite{Berge, Gabai89,Gabai90}, where it is shown that every such knot is a torus knot or a $1$-bridge braid. If we put solid torus in the standard position in $S^3$, these knots has nontrivial lens space surgeries. And as its name suggests, any lens space is an L-space. It is also proved that any $1$-bridge braid in the three-sphere or lens space is an L-space knot \cite{GLV}, and the L-spaces obtained by surgeries along $1$-bridge braids in $S^3$ has non-left-orderable fundamental groups \cite{Nie}. In line with these researches, we deduce similar properties of $(1,1)$ L-space knots in $S^3$.

\begin{theorem}\label{alternative-description}
A nontrivial positive $(1,1)$ L-space knot in $S^3$ can be represented as the closure of the braid
$$\left (\sigma_{\omega} \sigma_{\omega-1}\cdots \sigma_{\omega-b_0+1} \right)\left(\sigma_{\omega}\sigma_{\omega-1}\cdots \sigma_{1}\right)^{b_1}\left(\sigma_{\omega-1}\sigma_{\omega-2}\cdots \sigma_{1}\right)^{t-b_1}$$
in the braid group $B_{\omega+1}$ on $\omega+1$ strands, where $1\le b_0\le \omega$ and $1\le b_1\le t$.
\end{theorem}

An example is shown in Figure \ref{fig:braid} below.

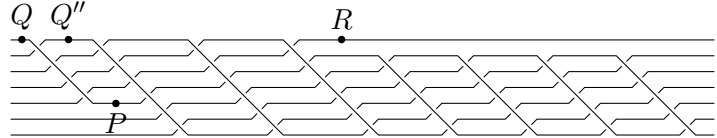
\begin{figure}[h]
\centering
\begin{tikzpicture}
    \coordinate[label = below:$P$] (P) at (1.4,0.42);

    \coordinate[label = $Q$] (Q) at (0.15,1.27);
    \coordinate[label = $Q''$] (Q2) at (0.77,1.27);

    \coordinate[label = $R$] (R) at (4.4,1.27);

\node at (P)[circle,fill,inner sep=1pt]{};
\node at (Q)[circle,fill,inner sep=1pt]{};
\node at (Q2)[circle,fill,inner sep=1pt]{};
\node at (R)[circle,fill,inner sep=1pt]{};
\pic[
rotate=90,
braid/.cd ,
gap=0.1,
control factor=0,
nudge factor=0,
height=-6pt,
width=6pt,
] { braid={
s_6^{-1} s_5^{-1} s_4^{-1} s_3^{-1}
s_6^{-1} s_5^{-1} s_4^{-1} s_3^{-1} s_2^{-1} s_1^{-1}
s_6^{-1} s_5^{-1} s_4^{-1} s_3^{-1} s_2^{-1} s_1^{-1}
s_6^{-1} s_5^{-1} s_4^{-1} s_3^{-1} s_2^{-1} s_1^{-1}
s_5^{-1} s_4^{-1} s_3^{-1} s_2^{-1} s_1^{-1}
s_5^{-1} s_4^{-1} s_3^{-1} s_2^{-1} s_1^{-1}
s_5^{-1} s_4^{-1} s_3^{-1} s_2^{-1} s_1^{-1}
s_5^{-1} s_4^{-1} s_3^{-1} s_2^{-1} s_1^{-1}
}};
\end{tikzpicture}
\caption{The braid when $(\omega,t,b_0,b_1)=(6,7,4,3)$.}\label{fig:braid}
\end{figure}

In \cite{Nie}, the author introduced the property (D) as follows.

\begin{definition}
For a nontrivial knot $K$ in $S^3$ with $\mu$ and $\lambda$ representing a meridian and a longitude in the knot group, we say $K$ has property (D) if 
\begin{enumerate}
    \item for any homomorphism $\rho$ from $\pi_1(S^3- K)$ to $\mbox{Homeo}_+(\mathbf{R})$, if $s\in\mathbf{R}$ is a common fixed point of $\rho(\mu)$ and $\rho(\lambda)$, then $s$ is a fixed point of every element in $\pi_1(S^3- K)$;
    \item $\mu$ is in the root-closed, conjugacy-closed submonoid generated by $\mu^{2g(K)-1}\lambda$ and $\mu^{-1}$.
\end{enumerate}
\end{definition}

The author proved that \cite[Theorem 1.3]{Nie} nontrivial knots which are closures of positive $1$-bridge braids have property (D). And by \cite[Theorem 4.1]{Nie}, it implies the non-left-orderability of the fundamental groups of the L-spaces obtained by Dehn surgeries on closures of $1$-bridge braids. In this paper, we prove the following result in a similar way. Thanks to the additional symmetry, our proof is simplified compared to the proof of \cite[Theorem 1.3]{Nie}.

\begin{theorem}\label{main-D}
Nontrivial positive $(1,1)$ L-space knots in $S^3$ have property (D).
\end{theorem}

Therefore, by \cite[Theorem 4.1]{Nie}, we have the following conclusion.

\begin{theorem}\label{non-left-orderable}
The fundamental group of an L-space obtained by Dehn surgery on a $(1,1)$-knot in $S^3$ is not left orderable.
\end{theorem}

Because a $(1,1)$-decomposition eases the computation of knot Floer homology, many examples of L-space knots which were studied in the literature are $(1,1)$-knots. Theorem \ref{non-left-orderable} serves as the generalization of relevant non-left-orderability results \cite{Chr16, Clay13, Ichi15, Ichi18,Jun,Liang, Nakae,Nie19,Nie, Tran19,Tran20}.

\section*{Acknowledgement}
The author thanks Fan Ye for helpful discussions.

	\section{Description of $(1,1)$ L-space knots}
	
	This section is dedicated to prove Theorem \ref{main-description}.

	Let $(\Sigma, \alpha, \beta,w,z)$ denote a reduced $(1,1)$-diagram representing an L-space knot in the three-sphere or lens space $Y$. Since the $\alpha$ curve is simple, it represents a primitive element in $H_1(\Sigma)$. We can straighten out the $\alpha$ curve via a self-homeomorphism of $\Sigma$. Assume that $\alpha$ curve is horizontal, with an orientation from left to right. The $\beta$ curve is cut into strands of two bands and two rainbows by the $\alpha$ curve. By \cite[Theorem 1.2]{GLV}, we can choose an orientation of $\beta$ which induces an orientation from left to right on every rainbow strand, which is the opposite of the coherent orientation. Assume that $w$ is on the left side of $\alpha$, and $z$ is on the right side of $\alpha$.
	
	\subsection{The first step}
	
	We define the positive curves on the torus $\Sigma$ as follows.

	\begin{definition}\label{positive}
	An oriented curve $\gamma$ on $\Sigma$ is called positive, if at each inner intersection point of $\gamma$ and $\alpha\cup \beta$, the $\gamma$ curve goes from the right side of the $\alpha$ or $\beta$ curve to the left side transversally.
	\end{definition}

	Our first step is to construct a positive curve $\gamma_0$ connecting $w$ to $z$. 

	Let $S$ be the set of endpoints of all positive curves originating from $w$. If $z \in S$, our first step is completed. Otherwise, we assume $z\not\in S$. Let $S_w$ be the connected component of $S-\alpha$ containing $w$. Since $w\in S_w$, each point on a rainbow strand around $w$ is an interior point of $S_w$. Since $z\not \in S$, each point on a rainbow strand around $z$ is in the exterior of $S_w$. Therefore, the shape of $S_w$ is a rectangle. Let the vertices of $S_w$ be $P_1,P_2,P_3,P_4$ counterclockwise, with $P_1P_2, P_4P_3$ being parts of the $\alpha$ curve, $P_2P_3, P_4P_1$ being parts of the $\beta$ curve, and the basepoint $w$ being close to the edge $P_1P_2$. From the definition of $S_w$, we can derive the orientation of $P_4P_1$ and $P_2P_3$ as shown in Figure \ref{fig:S_w}.

\begin{figure}[h]
\centering
\includegraphics[width=2in]{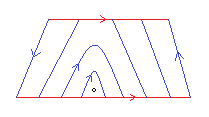}
\put(-74,22){$w$}
\put(-147,5){$P_1$}
\put(-10,5){$P_2$}
\put(-31,70){ $P_3$}
\put(-127,70){ $P_4$}
\caption{The rectangle $S_w$.}\label{fig:S_w}
\end{figure}
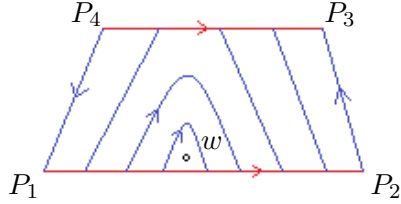

	If there exists another embedded open rectangle $P_4P_3P_5P_6$ on the left of $P_4P_3$ which does not contain any basepoints, then we replace $P_1P_2P_3P_4$ by the immersed rectangle $P_1P_2P_5P_6$ and try the same extension again. Because the $\beta$ curve is connected, the edge $P_2P_3$ extends to the right strand of the basepoint $z$ in finite steps. Hence, we assume that the sequence of extensions ends at an immersed rectangle $P_1P_2Q_2Q_1$, as shown in Figure \ref{fig:ppqq}.

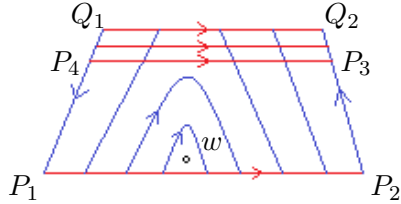
\begin{figure}[h]
\centering
\includegraphics[width=2in]{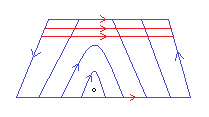}
\put(-74,22){$w$}
\put(-147,5){$P_1$}
\put(-10,5){$P_2$}
\put(-31,70){ $Q_2$}
\put(-127,70){ $Q_1$}
\put(-25,52){ $P_3$}
\put(-135,52){ $P_4$}
\caption{The immersed rectangle $P_1P_2Q_2Q_1$.}\label{fig:ppqq}
\end{figure}

	There are two possibilities for not able to extend the immersed rectangle: one of the edge $P_1Q_1$ and $P_2Q_2$ extends to a rainbow strand, or the embedded rectangle on the left of $Q_1Q_2$ contains at least one basepoint. If the embedded rectangle on the left of $Q_1Q_2$ contains the basepoint $z$. Then by the definition of $S$, we have $z\in S$. Otherwise, the edge $Q_1Q_2$ intersects with the edge $P_1P_2$ on the torus $\Sigma$.

	By the definition of $S_w$, the strands $P_4P_1$ and $P_2P_3$ are on the boundary of $S$, so we have $Q_1 Q_2 \subseteq P_1P_2$. If $P_1=Q_1$ or $P_2=Q_2$, then the edge $Q_1P_1$ or the edge $P_2Q_2$ covers the $\beta$ curve. In that case, $Q_1P_1$ or $P_2Q_2$ contains the left and right strand of the basepoint $z$, which is a contradiction. Therefore, the edge $Q_1Q_2$ lies in the interior of the edge $P_1P_2$.

	Suppose that there are $q$ rainbow strands in the middle, $r_1\ge 1$ band strands on the left (including $Q_1P_1$) and $r_2\ge1$ band strands on the right (including $P_2Q_2$) in the immersed rectangle $P_1P_2 Q_2 Q_1$. Suppose that the $i$-th intersection point on $Q_1Q_2$ is the $(i+k)$-th intersection point on $P_1P_2$ for $1\le i\le r_1+r_2$. Then we have $1\le k\le 2q-1$.

	For $1\le i \le 2q+r_1+r_2$, let $\varepsilon_i=1$ if the $\beta$ curve goes from the right side of the $\alpha$ curve to the left side at the $i$-th intersection point on $P_1P_2$. Otherwise, let $\varepsilon_i=-1$. Then we have 
$$
\varepsilon_i= 
\begin{cases}
\varepsilon_{i+k} &\mbox{, if }1\le i \le r_1;\\
1              &\mbox{, if }r_1+1\le i\le r_1+q;\\
-1             &\mbox{, if }r_1+q+1\le i\le r_1+2q;\\
\varepsilon_{i-2q+k} &\mbox{, if }2q+r_1+1\le i \le 2q+r_1+r_2.
\end{cases}
$$
If $1\le k\le q$, then we have $\varepsilon_1=1$ by induction. If $q+1\le k\le 2q-1$, then we have $\varepsilon_{2q+r_1+r_2}=-1$ by induction. Either case leads to a contradiction.
	
	\subsection{The second step}

	We have constructed a positive curve $\gamma_0$ connecting $w$ to $z$. By eliminating self-loops, we assume that $\gamma_0$ is simple and intersects each connected component of $\Sigma-\alpha-\beta$ at most once. Our second step is to construct a positive simple closed curve $\gamma$ passing through $w$ and $z$. 

	Let $T_1, T_2,\ldots, T_p$ be all intersection points between the $\alpha$ curve and the $\beta$ curve, ordered along the orientation of $\alpha$. Via a self-homeomorphism of $\Sigma$, we assume the following condition: for $1\le i\le p$, if the $\alpha$-segment $T_i T_{i+1}$ does not intersect with the $\gamma_0$ curve, then it has unit length; otherwise, it has length $2q+1$, where $q$ is the number of strands in each rainbow. 

	For a downward-oriented band strand $e_1$ and an upward-oriented band strand $e_2$ on the $\beta$ curve, there exists an embedded open rectangle $R$ with two edges being $e_1$ and $e_2$ and the other two edges $e_3,e_4$ on the $\alpha$ curve. We can further assume that the rectangle is on the left of $e_1, e_2, e_3$ and on the right of $e_4$. 

	Let $l_i$ denote the length of $e_i$ for $i=3,4$, then $$l_i=\left|e_i\cap \beta \right|+ 2q \left |e_i\cap\gamma_0 \right|-1.$$ The difference $\left|e_4\cap \beta\right|-\left|e_3\cap \beta\right|$ depends on whether each basepoint lies in $R$, that is, $$\left|e_4\cap \beta\right|-\left|e_3\cap \beta\right|=2q\left|\{z\}\cap R\right|-2q\left|\{w\}\cap R\right|.$$ The difference $\left|e_4\cap \gamma_0\right|-\left|e_3\cap \gamma_0\right|$ depends on how $\gamma_0$ intersects $R$. Since $\gamma_0$ is positive, if it intersects $e_1$, $e_2$ or $e_3$ at a point, then it enters $R$ there; if it intersects $e_4$ at a point, then it exits $R$ there. Hence we have $$ \left|e_4\cap \gamma_0\right|-\left|e_3\cap \gamma_0\right|=\left|e_1\cap \gamma_0\right|+\left|e_2\cap \gamma_0\right| +\left|\{w\}\cap R\right|-\left|\{z\}\cap R\right|.$$ By combining these equations, we get $l_3\le l_4$.

	Therefore, there exists a linear foliation $\mathcal{F}$ of the torus $\Sigma$, such that, up to isotopy, each strand of the $\beta$ curve is contained in a leaf of $\mathcal{F}$ or transverse to $\mathcal{F}$ in a fixed direction. Via another isotopy, we can assume that the entire $\beta$ curve is either contained in a leaf of $\mathcal{F}$ or transverse to $\mathcal{F}$. In either case, we can assume that the slope of $\mathcal{F}$ is irrational under a perturbation of the foliation, so the leaves of $\mathcal{F}$ are dense. We extend the curve in both directions from the basepoint $w$ along a leaf of $\mathcal{F}$ until the endpoints reach the connected component of $\Sigma-\alpha-\beta$ containing the basepoint $z$. After closing the curve by connecting two endpoints within the connected component, we get the positive simple closed curve $\gamma$ passing through $w$ and $z$. 
	
	\subsection{The third and the last steps}\label{ss3}
	Our third step is to complete the proof of the ``only if'' part of Theorem \ref{main-description}. Via a self-homeomorphism of $\Sigma$, we abandon the horizontality of the $\alpha$ curve, and assume that the $\gamma$ curve is horizontal instead. Since $\gamma$ is positive, we can either assume $\alpha$ is a geodesic or assume $\beta$ is a geodesic, but not simultaneously. In fact, there exists an isotopy $f:(\alpha\cup \beta\cup \gamma)\times [0,1]\to \Sigma$, such that $f(z,t), f(\gamma,t)$ are independent of $t$, and $f(\alpha,0), f(\beta,1), f(\gamma,0)$ are geodesics. Let $\rho$ be the curve in $\Sigma\times[0,1]$ defined by $\rho(t)=(f(w,t),t)$. Let $\tau_0$ (resp. $\tau_1$) be a geodesic in $(\Sigma,0)$ (resp. $(\Sigma,1)$) which does not intersect with $(f(\alpha,0),0)$ (resp. $(f(\beta,1),1)$). After attaching the solid tori $U_0$ and $U_1$ to the boundary components of $\Sigma\times[0,1]$, such that $(f(\alpha,0),0)$ (resp. $(f(\beta,1),1)$) is a meridional disk of $U_0$ (resp. $U_1$), we recover the knot $\rho\cup \tau_0 \cup \tau_1$ from the $(1,1)$-diagram.

	At last, the ``if'' part of Theorem \ref{main-description} can be proved in a way similar to the $1$-bridge braid case. In \cite[Section 3]{GLV}, a coherent reduced $(1,1)$-diagram was constructed for each $1$-bridge braid in $S^3$ and lens space to utilize \cite[Theorem 1.2]{GLV}, as shown in Figure \ref{fig:GLV}. We make a tiny change in the construction here: the basepoint $z$ is no longer restricted to be the starting point of the $\gamma'$, but can be any point in $\Sigma-\alpha-\beta$. The topological meaning of the diagram is as explained in the previous paragraph: if we move the basepoint $w$ along a geodesic $\gamma'$, we can untwist the $\beta$ curve at the expense of twisting the $\alpha$ curve. With this change, a $(1,1)$-diagram as the middle one in Figure \ref{fig:GLV} can represent the knot described in Theorem \ref{main-description}. This $(1,1)$-diagram is coherent in the sense that certain orientations on the $\alpha$ and $\beta$ curves induce coherent orientations on the boundary of every embedded bigon $(D,\partial D)\subseteq (\Sigma, \alpha \cup \beta)$. Each isotopy to remove an empty bigon preserves the coherence, so we get a reduced $(1,1)$-diagram in finite steps. By \cite[Theorem 1.2]{GLV}, we proved the ``if'' part of Theorem \ref{main-description}.

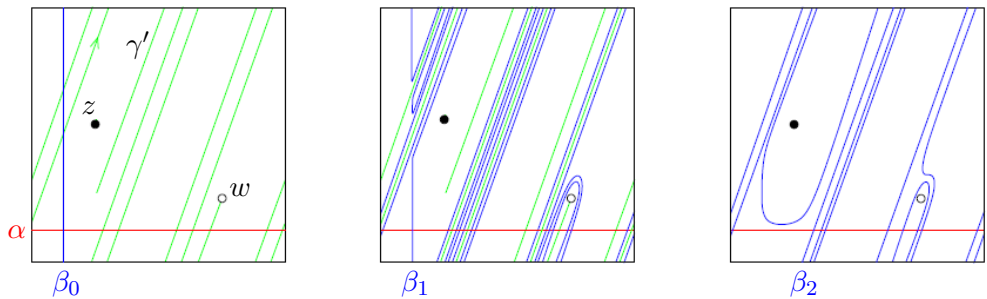
\begin{figure}[h]
\centering
\includegraphics[width=5in]{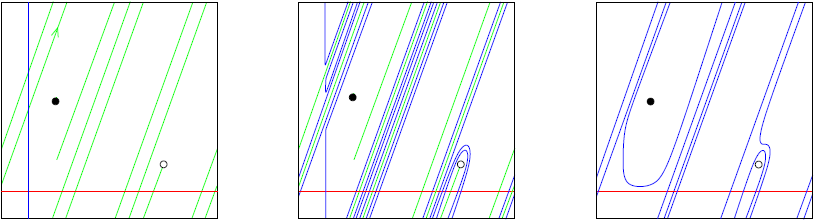}
\put(-342,57){$z$}
\put(-286,26){$w$}
\put(-370,10){\color{red} $\alpha$}
\put(-325,80){$\gamma'$}
\put(-353,-10){\color{blue} $\beta_0$}
\put(-221,-10){\color{blue} $\beta_1$}
\put(-74,-10){\color{blue} $\beta_2$}
\caption{The construction of a coherent diagram of the $1$-bridge braid $K(-2,3,7)$ in $S^3$, modified from \cite[Figure 3]{GLV}.}\label{fig:GLV}
\end{figure}

\section{Non-left-orderable surgeries}
\subsection{A positive braid representation}
In this subsection, we prove Theorem \ref{alternative-description} and derive the genus formula.

Let $K$ be a nontrivial positive $(1,1)$ L-space knot in $S^3$. Let $S^3=U_0\cup_{\Sigma} U_1$ be a genus one Heegaard splitting with standard geometry. By Theorem \ref{main-description}, $K$ is isotopic to $\rho \cup \tau_0 \cup \tau_1$, where $\rho$ is a geodesic of $\Sigma$, and $\tau_0$ (resp. $\tau_1$) is properly embedded in some meridional disk of $U_0$ (resp. $U_1$). 

An orientation on the geodesic $\rho$ induces orientations on the cores of $U_0$ and $U_1$. If the cores of $U_0$ and $U_1$ are negatively linked, then the construction in Subsection \ref{ss3} yields a negative coherent reduced $(1,1)$-diagram. By \cite[Theorem 1.2]{GLV}, $K$ is a negative $(1,1)$ L-space knot, which contradicts the assumption that $K$ is a nontrivial positive $(1,1)$ L-space knot in $S^3$. Thus, the cores of $U_0$ and $U_1$ with induced orientations are positively linked. So $\rho$ can be realized as a part of a positive braid. After appending the arcs $\tau_0$ and $\tau_1$, we get a positive braid as shown in Figure \ref{fig:braid}.

Let $K$ be the closure of the positive braid represented by $$\left (\sigma_{\omega} \sigma_{\omega-1}\cdots \sigma_{\omega-b_0+1} \right)\left(\sigma_{\omega}\sigma_{\omega-1}\cdots \sigma_{1}\right)^{b_1}\left(\sigma_{\omega-1}\sigma_{\omega-2}\cdots \sigma_{1}\right)^{t-b_1}.$$
If $b_0=b_1=0$, then $K$ has an unknot component, which is not allowed. If $b_0=0$, we can decrease $t$ and $b_1$ by one and set $b_0$ to $\omega$. If $b_1=0$, we can decrease $\omega$ and $b_0$ by one and set $b_1$ to $t$. For the representation with minimal $t+\omega$, we have $1\le b_0\le \omega$ and $1\le b_1\le t$. Therefore, Theorem \ref{alternative-description} holds.

A minimal genus Seifert surface is obtained \cite{Cromwell} by applying Seifert's algorithm to a positive diagram, so the genus of $K$ is
\begin{align*}
g(K)&=\frac{1}{2}(\mbox{\#crossings}-\mbox{\#strands}+1)\\
&=\frac{1}{2}(b_0+b_1\omega +(t-b_1)(\omega-1)- (\omega+1)+1)\\
&=\frac{1}{2}(t \omega-t-\omega+b_0+b_1).
\end{align*}

\subsection{The knot group}

In this subsection, we investigate the knot group $\pi_1(S^3-K)$. As a $(1,1)$-knot, the knot group has a $2$-generator presentation. However, to keep the symmetry, we specify four elements $x_0,y_0,x_1,y_1$ in the knot group instead. 

Let $D_0$ (resp. $D_1$) be the meridional disk of $U_0$ (resp. $U_1$) containing $\tau_0$ (resp. $\tau_1$). Then $D_0$ (resp. $D_1$) is divided by $\tau_0$ (resp. $\tau_1$) into two disks $D_{x,0}$ and $D_{y,0}$ (resp. $D_{x,1}$ and $D_{y,1}$). Let the points $P, Q, R$ on $\Sigma$ be $\rho\cap \tau_0, \tau_0\cap \tau_1, \tau_1\cap \rho$, respectively. Orient the knot $K$ so that $P, Q, R$ appears in order. Orient the cores of $U_0, U_1$ and the disks $D_0,D_1,D_{x,0},D_{y,0},D_{x,1},D_{y,1}$ accordingly. Let $Q'$ be a point near $Q$ in $\Sigma-\rho-D_0-D_1$, so that $Q'$ is on the negative side of $D_0$ and on the positive side of $D_1$. Let $Q''$ be a point in $\rho$ on the boundary of the connected component of $\Sigma-\rho-D_0-D_1$ containing $Q'$, as shown in Figure \ref{fig:braid}.

The fundamental group of $U_0- \tau_0$ (resp. $U_1-\tau_1$) is freely generated by two elements $x_0,y_0$ (resp. $x_1,y_1$), where $x_0$ (resp. $y_0,x_1,y_1$) is represented by a loop based at $Q'$ intersecting $D_{x,0}$ (resp. $D_{y,0},D_{x,1},D_{y,1}$) once positively and not intersecting other disks. Then $\pi_1(S^3-K)$ based at $Q'$ is generated by $x_0,y_0,x_1,y_1$.

Without loss of generality, we assume $x_0$ (resp. $x_1$) has larger norm than $y_0$ (resp. $y_1$) in $H_1(S^3-K)$. Then $$\mu = x_0 y_0^{-1}=y_1^{-1}x_1$$
represents a meridian of $K$ around $Q$.

The boundary of $D_1$ (resp. $D_0$) intersects $\rho$ in $t$ (resp. $\omega$) points, not counting $P$, $Q$ and $R$. Starting from $Q$ along positive direction, let the points be $R_t,R_{t-1},\ldots, R_1$ (resp. $P_\omega, P_{\omega-1}, \ldots, P_1$) in order. For each $i$ with $1\le i\le t$ (resp. $1\le i\le \omega$), let $g_i$ (resp. $h_i^{-1}$) represent the loop based at $Q'$ in $U_0-\tau_0$ (resp. $U_1-\tau_1$) which first travels to $R_i$ (resp. $P_i$) without intersecting $D_0$ (resp. $D_1$), then follows $\rho$ in positive (resp. negative) direction to $P$ (resp. $R$) but not past it, lastly travels back to $Q'$ without intersecting $D_0$ (resp. $D_1$). Then each $g_i$ (resp. $h_i$) represented by a word in $x_0$ and $y_0$ (resp. $x_1$ and $y_1$), and we have 
\begin{align*}
y_0&=\left(g_1\mu g_1^{-1} \right)\left(g_2\mu g_2^{-1} \right)\cdots \left(g_t \mu g_t^{-1}\right),\\
y_1&=\left(h_1^{-1}\mu h_1 \right)\left(h_2^{-1}\mu h_2 \right)\cdots \left(h_\omega^{-1} \mu h_\omega\right).
\end{align*}
Since $b_0,b_1\neq 0$, the point $Q''$ is on the arc $R_1 P_\omega\subset \rho$ which is a part of boundary of the connected component of $\Sigma-\rho-D_0-D_1$ containing $Q'$. The longitude $\lambda$ of $K$ starting from $Q$ is determined by
$$\mu^{k_0} \lambda= h_\omega g_1.$$
The integer $k_0$ can be found by counting the crossings between $K$ and a loop represented by $h_\omega g_1$ on a planar diagram. The loop represented by $h_\omega g_1$ differs from the blackboard framing of $K$ as shown in Figure \ref{fig:braid} by $2t$ additional positive crossings, so we have
\begin{align*}
k_0&=\mbox{\#crossings}+t\\
   &= b_0+b_1\omega +(t-b_1)(\omega-1)+t\\ 
   &= t \omega+b_0+b_1.
\end{align*}
Therefore we have 
$$\mu^{t \omega+b_0+b_1} \lambda= h_\omega g_1,$$
and
$$\mu^{2g(K)-1} \lambda= h_\omega g_1 \mu^{-t-\omega -1 }.$$

Furthermore, the word representing $g_1$ starts with an $x_0$, and the word representing $h_\omega$ ends with an $x_1$.

\subsection{The property (D)}
In this subsection, we prove that $K$ has property (D).

The first part of the property (D) is the following.

\begin{lemma}\label{D-1}
For any homomorphism $\rho$ from $\pi_1(S^3- K)$ to $\mbox{Homeo}_+(\mathbf{R})$, if $s\in\mathbf{R}$ is a common fixed point of $\rho(\mu)$ and $\rho(\lambda)$, then $s$ is a fixed point of every element in $\pi_1(S^3- K)$.
\end{lemma}
\begin{proof}
Since $s$ is a common fixed point of $\rho(\mu)$ and $\rho(\lambda)$, it is a common fixed point of $\rho(x_0 y_0^{-1})$, $\rho(y_1^{-1}x_1)$ and $\rho(h_\omega g_1)$. Without loss of generality, we assume $\rho(x_0)s\ge s$, then we have $\rho(y_0)s\ge s$. We also have $\rho(x_1)s\ge s$ (resp. $\rho(x_1)s\le s$) if and only if $\rho(y_1)s\ge s$ (resp. $\rho(y_1)s\le s$).

Starting from the base point $Q'$, we construct a geodesic $\gamma$ on $\Sigma-\rho$ parallel to $\rho$. Because $K$ is nontrivial, the arc $\rho$ is not parallel to $\partial D_0$ or $\partial D_1$. Extend $\gamma$  until it crosses each disk $D_0, D_1$ at least once and reaches the connected component of $\Sigma-\rho-D_0-D_1$ containing $Q'$ again. Then we close up the curve to obtain the knot group element $g_0$, which can be represented by a nontrivial word in $x_0$ and $y_0$, and also by a nontrivial word in $x_1$ and $y_1$. By the first condition, we have $\rho(g_0)s\ge s$. By the second condition, we have $\rho(x_1)s\ge s$ and $\rho(y_1)s \ge s$. Because $h_\omega g_1$ is represented by a word in $x_0,y_0,x_1,y_1$ with at least one $x_0$ and one $x_1$, we have $\rho(x_0)s=\rho(y_0)s =\rho(x_1) s =\rho(y_1)s= s$. Therefore $s$ is a fixed point of every element in $\pi_1(S^3- K)$.
\end{proof}

\begin{remark}
This lemma is the starting point of our research. Suppose we are given a $(1,1)$ L-space knot in the form of a coherent diagram. Then the knot group has two generators $x_0,y_0$ and one relation. To prove this lemma, we have to construct another element such as $y_1$ in terms of $x_0$ and $y_0$. The algebraic construction is highly nontrivial, and better to be done topologically as Theorem \ref{main-description}.
\end{remark}

Then we prove the second part of the property (D).

\begin{lemma}\label{D-2}
The element $\mu$ is in the root-closed, conjugacy-closed submonoid generated by $\mu^{2g(K)-1}\lambda$ and $\mu^{-1}$.
\end{lemma}

\begin{proof}
As in \cite[Section 3]{Nie}, we define the preorder $\le_k $ generated by $\mu$ and $(\mu^{2g(K)-1}\lambda)^{-1}$ on $\pi_1(S^3- K)$. Since $\mu =  x_0 y_0^{-1}=y_1^{-1}x_1$, we have $x_0\ge_k y_0$ and $x_1\ge_k y_1$. Since $\mu^{2g(K)-1} \lambda= h_\omega g_1 \mu^{-t-\omega -1 }$, we have $h_\omega g_1\le_k \mu^{t+\omega+1}$.

Let $\tilde{g}_0=1, \tilde{g}_1, \ldots, \tilde{g}_{t'}=g_1$ be all suffixes of $g_1$, and $\tilde{h}_0, \tilde{h}_1, \ldots, \tilde{h}_{\omega'}=h_\omega$ be all prefixes of $h_\omega$, ordered by length. Suppose that $\tilde{g}_i$ appears $m_i$ times in $g_1,g_2,\ldots, g_t$ for each $0\le i< t'$, and $\tilde{h}_i$ appears $n_i$ times in $h_1,h_2,\ldots, h_\omega$ for each $0\le i< \omega'$. Then we have 
\begin{align*}
y_0&=\left(g_1\mu g_1^{-1} \right)\left(g_2\mu g_2^{-1} \right)\cdots \left(g_t \mu g_t^{-1}\right)\\
	&\ge_k \left(g_1\mu g_1^{-1}\right)\left(\tilde{g}_{i} \mu \tilde{g}_{i}^{-1}\right)^{m_i},
\end{align*}
and
\begin{align*}
y_1&=\left(h_1^{-1}\mu h_1 \right)\left(h_2^{-1}\mu h_2 \right)\cdots \left(h_\omega^{-1} \mu h_\omega \right)\\
	&\ge_k \left(\tilde{h}_{i}^{-1} \mu \tilde{h}_{i}\right)^{n_i} \left(h_\omega^{-1} \mu h_\omega \right).
\end{align*}

For each $0\le i< t'$, we have either $\tilde{g}_{i+1}=y_0\tilde{g}_{i}$ or $\tilde{g}_{i+1}=x_0\tilde{g}_{i}=\mu y_0\tilde{g}_{i}$. And for $i=t'-1$, it is necessarily the latter case. So we have 
$$\tilde{g}_{i+1}\ge_k y_0 \tilde{g}_i\ge_k \left(g_1\mu g_1^{-1}\right) \tilde{g}_{i} \mu^{m_i}.$$
By induction, we have 
$$ g_1=\tilde{g}_{t'}=\mu y_0 \tilde{g}_{t'-1}\ge_k \mu g_1 \mu^{t'} g_1^{-1}\mu^{\sum_{i=0}^{t'-1} m_i}.$$
By symmetry, we have 
$$ h_\omega\ge_k \mu^{\sum_{i=0}^{\omega'-1} n_i} h_\omega^{-1} \mu^{\omega'} h_\omega\mu.$$

Here $t'$ (resp. $\omega'$) is the number of intersection points between $Q'' P\subset \rho$ and $D_0$ (resp. $RQ''\subset \rho$ and $D_1$), not counting $P$ and $R$. And $\sum_{i=0}^{t'-1} m_i$ (resp. $\sum_{i=0}^{\omega'-1} n_i$) is the number of intersection points between $Q'' P\subset \rho$ and $D_1$ (resp. $RQ'' \subset \rho$ and $D_0$). So we have $$\sum_{i=0}^{t'-1} m_i=\omega-\omega' $$
and
$$\sum_{i=0}^{\omega'-1} n_i=t-t'.$$

Then we have
\begin{align*}
h_\omega&\ge_k \mu^{t-t'} h_\omega^{-1} \mu^{\omega'} h_\omega \mu\\
&\ge_k \mu^{t-t'} h_\omega^{-1} \mu^{\omega'} h_\omega,
\end{align*}
So $h_\omega\ge_k \mu ^{t-t'+\omega'}$. Because $h_\omega g_1\le_k \mu^{t+\omega+1}$, we have $g_1\le_k\mu^{\omega+t'-\omega'+1}$. Then we have
\begin{align*}
g_1&\ge_k \mu g_1 \mu^{t'}g_1^{-1} \mu^{\omega-\omega'}\\
&\ge_k \mu g_1 \mu^{-1}
\end{align*}
Since $g_1 \ge_k \mu g_1 \mu^{-1}$, we get $g_1 \mu^{t'} g_1^{-1}\ge_k \mu^{t'}$. So we have
\begin{align*}
g_1&\ge_k \mu g_1 \mu^{t'}g_1^{-1} \mu^{\omega-\omega'}\\
&\ge_k \mu^{\omega+t'-\omega'+1}.
\end{align*}
By symmetry, we have $h_\omega \ge_k \mu^{t-t'+\omega'+1}$. By $h_\omega g_1\le_k \mu^{t+\omega+1}$, we have $\mu\le_k 1$. In other words, the meridian $\mu$ is in the root-closed, conjugacy-closed submonoid generated by $\mu^{2g(K)-1}\lambda$ and $\mu^{-1}$.
\end{proof}

Combining Lemma \ref{D-1} and Lemma \ref{D-2}, we proved Theorem \ref{main-D}. By \cite[Theorem 4.1]{Nie}, we proved Theorem \ref{non-left-orderable}.

\end{document}